\documentclass[11pt]{article}
\usepackage{amsmath,amssymb,amsfonts,textcomp,stmaryrd,amsthm,xifthen,psfrag,graphicx,color}
\usepackage{color}

\SetSymbolFont{stmry}{bold}{U}{stmry}{m}{n} 
\usepackage[T1]{fontenc}
\date{}

\newtheorem{theorem}{Theorem}
\newtheorem{lemma}[theorem]{Lemma}

\newtheorem{remark}[theorem]{Remark}
\theoremstyle{definition} 

\newcommand{\dual}[2]{\langle#1\hspace*{.5mm},#2\rangle}
\newcommand{\vdual}[2]{(#1\hspace*{.5mm},#2)}
\newcommand{\abs}[1]{\vert #1 \vert}
\newcommand{\norm}[3][]{#1\|#2#1\|_{#3}}

\newcommand{\diam}{\mathrm{diam}}
\newcommand{\wilde}{\widetilde}
\newcommand{\wat}{\widehat}
\def\pwnabla{\nabla_\TT}

\def\div{{\rm div\,}}
\def\pwdiv{ {\rm div}_{\TT}\,}

\newcommand{\trace}{\gamma}
\newcommand{\jump}[1]{[#1]}
\newcommand{\hp}{{hp}}
\def\uc{u^c}
\def\Omc{\Omega^c}

\newcommand{\opt}{{\rm opt}}

\newcommand{\R}{\ensuremath{\mathbb{R}}}
\newcommand{\N}{\ensuremath{\mathbb{N}}}
\newcommand{\HH}{\ensuremath{\mathbf{H}}}

\newcommand{\LL}{\ensuremath{\mathbf{L}}}

\newcommand{\nn}{\ensuremath{\mathbf{n}}}
\newcommand{\vv}{\ensuremath{\mathbf{v}}}
\newcommand{\TT}{\ensuremath{\mathcal{T}}}
\newcommand{\tT}{\ensuremath{\mathbf{T}}}
\newcommand{\cS}{\ensuremath{\mathcal{S}}}
\newcommand{\sS}{\ensuremath{\mathbf{S}}}
\newcommand{\el}{\ensuremath{T}}
\newcommand{\ed}{\ensuremath{e}}

\newcommand{\OO}{\ensuremath{\mathcal{O}}}
\newcommand{\EE}{\ensuremath{\mathcal{E}}}

\newcommand{\GG}{\ensuremath{\mathbf{G}}}

\newcommand{\ssigma}{{\boldsymbol\sigma}}
\newcommand{\ttau}{{\boldsymbol\tau}}
\newcommand{\qq}{{\boldsymbol{q}}}
\newcommand{\uu}{\boldsymbol{u}}

\newcommand{\n}{\boldsymbol{n}}

\newcommand{\slo}{\mathcal{V}}
\newcommand{\slp}{\widetilde\slo}

\newcommand{\dlo}{\mathcal{K}}
\newcommand{\dlp}{\widetilde\dlo}
\newcommand{\adlo}{\mathcal{K}'}

\newcommand{\jn}{\text{jn}}
\newcommand{\bm}{\text{bm}}

\newcounter{constantsnumber}
\def\setc#1{
  \ifthenelse{\equal{#1}{poinc}}{C_{\rm edge}}{ 
   \refstepcounter{constantsnumber}
   \label{const#1}C_{\theconstantsnumber}}}

\def\c#1{
  \ifthenelse{\equal{#1}{poinc}}{C_{\rm edge}}{ 
    C_{\ref{const#1}}}}

\title{DPG method with optimal test functions\\for a transmission problem
\thanks{Supported by CONICYT through FONDECYT projects 1110324, 3140614
        and Anillo ACT1118 (ANANUM).}}
\author{
Norbert Heuer$^\dagger$
\and
Michael Karkulik\thanks{
Facultad de Matem\'aticas, Pontificia Universidad Cat\'olica de Chile,
Avenida Vicu\~na Mackenna 4860, Macul, Santiago, Chile,
email: {\tt \{nheuer,mkarkulik\}@mat.puc.cl}}}

\begin{document}
\maketitle
\begin{abstract}
  We propose and analyze a numerical method to solve an elliptic transmission problem in full space.
  The method consists of a variational formulation
  involving standard boundary integral operators on the coupling interface
  and an ultra-weak formulation in the interior.
  To guarantee the discrete inf-sup condition, the system is discretized
  by the DPG method with optimal test functions.
  We prove that principal unknowns are approximated quasi-optimally.
  Numerical experiments for problems with smooth and singular solutions
  confirm optimal convergence orders.


\end{abstract}
\section{Introduction}
In recent years, Demkowicz and Gopalakrishnan have established the
\textit{discontinuous Petrov-Galerkin method with optimal test functions}
as a method that is designed to be stable \cite{DemkowiczG_11_CDP,dg11}.
In particular, it aims at robust discretizations of singularly perturbed problems,
cf.~\cite{BroersenS_14_IMA,BroersenS_14_CAMWA,CaloCN_14,ChanHBD_14,DemkowiczH_13}.
In general terms, this method consists in applying Petrov-Galerkin approximations with optimal test functions
to (in most cases) ultra-weak variational formulations (see~\cite{CessenatD_98_UWF} for an early use of
ultra-weak formulations). Such a variational formulation is obtained by element-wise integration
by parts on some partition and the replacement of appearing terms on the elements' boundaries with new
unknowns (see~\cite{BottassoMS_02_DPG} for the idea of introducing independent boundary unknowns).

Until recently the DPG method with optimal test functions has been studied only for problems on bounded domains.
In~\cite{HeuerK_14_DPGBEM}, we considered boundary value and screen problems of Neumann type which can be reduced
to a hypersingular boundary integral equation. This includes the case of a PDE on an unbounded domain.
In this paper we study for the first time a DPG strategy with optimal test functions for a transmission problem.
This problem is of academic nature (Poisson equation) and set in the full space.
We expect that our fairly general approach is applicable to more practical cases of
transmission problems with singular perturbations on bounded subdomains. This is feasible whenever there
is a DPG finite element technique available for the corresponding singularly perturbed problem on the subdomain.

Transmission problems often appear in the modeling of multiple physical phenomena, and therefore require
the combination of two, possibly different, numerical methods.
A popular approach is the coupling of finite elements and boundary elements.
Whereas finite elements can be used in a relatively straightforward way to solve nonlinear PDEs with
space-dependent coefficients and sources, it is a challenging task to apply them on unbounded domains.
On the other hand, boundary elements can deal naturally with unbounded domains but, by construction,
work best for linear, homogeneous PDEs with constant coefficients.

There are different approaches for the coupling of finite and boundary elements. In this paper we consider
the so-called nonsymmetric or one-equation coupling, also referred to as Johnson-N\'ed\'elec coupling.
It has the practical advantage of involving only two of the four classical boundary integral operators.
This coupling was mathematically analyzed first in~\cite{BrezziJ_79_FEMBEM,BrezziJN_78_FEMBEM,JohnsonN_80_FEMBEM}.
The mathematical proofs in the mentioned works require one of the involved boundary integral operators
(the double-layer operator) to be compact. In case of the Laplace equation this property does not hold on
polyhedral domains and, in the case of linear elasticity, not even on domains with smooth boundary.
In~\cite{s13}, Sayas proved the well-posedness of the Johnson-N\'ed\'elec coupling without using
compactness arguments and thus gave a mathematical justification of the nonsymmetric coupling even on
polyhedral domains. Since then, different authors have re-considered nonsymmetric couplings,
see~\cite{affkmp13,ffkp14,ghs12,OfS_14_FEMBEM,Steinbach_11_FEMBEM}. For an extensive discussion of
this topic, we refer to~\cite{f14}.

In this paper, we extend the nonsymmetric coupling of Johnson-N\'ed\'elec to a DPG method with optimal test functions.
More specifically, given a coupled system of PDEs (one in a bounded domain, one in the exterior),
we use an ultra-weak formulation for the interior part coupled to classical boundary integral equations
for the exterior problem. The whole system is discretized by the Petrov-Galerkin method with optimal test functions.

An alternative approach would be to couple standard boundary elements with a DPG scheme restricted to the
interior problem, i.e., optimal test functions are only used for the interior (finite element) part.
In this case, however, several difficulties arise. For instance, it is unclear how to choose the remaining test
functions to generate a square system. The design and analysis of a DPG-BEM coupling is left for future research.
In contrast, the approach of computing optimal test functions for the whole system (as in this paper)
appears to be the most generic one and deserves a thorough analysis.

An outline of this paper is as follows. In Section~\ref{section:math}, we
introduce the model problem and present the mathematical framework.
We also summarize some results related to the DPG method, to be used subsequently.
In Subsection~\ref{section:results} we formulate the method and state the main results.
Theorem~\ref{thm:solve} establishes stability of the continuous variational formulation
and Theorem~\ref{thm:opt} shows the quasi-optimality of conforming discretizations (so-called C\'ea-lemma).
Proofs of the two theorems are given in Section~\ref{sec:proofs}. Principal part of these
proofs involves the analysis of adjoint problems. This analysis is made in Section~\ref{section:adjoint}.
In Section~\ref{section:numerics} we present some numerical results that underline our theory.
Some conclusions are drawn in Section~\ref{section:conclusio}.
\section{Mathematical setting and main results}\label{section:math}
Let $\Omega\subset\R^d$, $d\geq2$, be a bounded Lipschitz domain with boundary $\Gamma$.
Our model transmission problem is as follows:
given volume data $f$ and jumps $u_0$, $\phi_0$, find $u$ and $u^c$ such that
\begin{subequations}\label{eq:tp}
\begin{align}
  -\div \nabla u &= f \text{ in } \Omega\label{eq:tp:a}\\
  -\div \nabla \uc &= 0 \text{ in } \Omc := \R^d\setminus\overline\Omega\label{eq:tp:b}\\
  u-\uc &= u_0 \text{ on } \Gamma\label{eq:tp:c}\\
  \frac{\partial}{\partial\n_\Omega}(u-\uc) &= \phi_0 \text{ on } \Gamma\label{eq:tp:d}\\
  \uc(x) &= \OO(\abs{x}^{-1}) \text{ as } \abs{x}\rightarrow\infty.\label{eq:tp:e}
\end{align}
\end{subequations}
The normal vector $\n_\Omega$ on $\Gamma$ points in direction of $\Omc$. 
Here, $f\in L_2(\Omega)$, $u_0\in H^{1/2}(\Gamma)$, $\phi_0\in H^{-1/2}(\Gamma)$, and the
spaces are of standard Sobolev type (some more details are given in the next section).
For $d=2$, we assume that $\diam(\Omega)<1$ and $\int_\Omega f + \int_\Gamma \phi_0 = 0$.
The scaling condition on $\Omega$ is to ensure the ellipticity of
the single layer operator, and the compatibility condition on the data $f$ and $\phi_0$
is needed in order to use the radiation condition in the form~\eqref{eq:tp:e}.
\subsection{Abstract DPG method}\label{section:dpg}
We briefly recall the abstract framework of the DPG method with optimal test functions,
cf.~\cite{dg11,DemkowiczG_11_CDP,zmdgpc11}. We state this in a
form that will be convenient for the forthcoming analysis.
The continuous framework is provided by the following result which
is a consequence of the open mapping theorem and the properties
of conjugate operators, cf.~\cite[Chapters~II.5,~VII.1]{y95}.
In this manuscript, all suprema are taken over the indicated sets \textit{except} $0$.
\begin{lemma}\label{lem:basic}
  Denote by $U$ and $V$ two reflexive Banach spaces. Let
  $B:U\rightarrow V'$ be a bijective and bounded linear operator with conjugate operator
  $B':V\rightarrow U'$. Define
  \begin{align*}
    \norm{\uu}{U,\opt} := \sup_{\vv\in V}\frac{\dual{B\uu}{\vv}}{\norm{\vv}{V}}
    \quad\text{ and }\quad \norm{\vv}{V,\opt} := \sup_{\uu\in U}\frac{\dual{B\uu}{\vv}}{\norm{\uu}{U}}.
  \end{align*}
  Then both operators $B$ and $B'$ are isomorphisms, and
  $\norm{\cdot}{U,\opt}$ and $\norm{\cdot}{V,\opt}$ define norms in $U$ and $V$ which are
  equivalent to $\norm{\cdot}{U}$ and $\norm{\cdot}{V}$, respectively.
  Furthermore, there holds
  \begin{align}\label{lem:basic:norms}
    \norm{\uu}{U} = \sup_{\vv\in V}\frac{\dual{B\uu}{\vv}}{\norm{\vv}{V,\opt}}
    \quad\text{ and }\quad
    \norm{\vv}{V} = \sup_{\uu\in U}\frac{\dual{B\uu}{\vv}}{\norm{\uu}{U,\opt}}.
  \end{align}
\end{lemma}
For a proof of~\eqref{lem:basic:norms} we refer to~\cite[Prop.~2.1]{zmdgpc11}.
Now suppose additionally that $V$ is a Hilbert space.
Given a bilinear form $b:U\times V\rightarrow \R$, and
a linear functional $\ell\in V'$, we aim to
\begin{align}\label{eq:var}
  \text{find } \uu\in U \text{ such that } b(\uu,\vv)=\ell(\vv) \quad\text{ for all } \vv\in V.
\end{align}
For an approximation space $U_\hp\subset U$, the Petrov-Galerkin method with optimal
test functions is to
\begin{align}\label{eq:dpg}
  \text{find } \uu_\hp\in U_\hp \text{ such that } b(\uu_\hp,\vv_\hp) = \ell(\vv_\hp)\quad\forall \vv_\hp\in \Theta(U_\hp).
\end{align}
Here, $\Theta:\;U\to V$ is the trial-to-test operator defined by
\begin{align*}
   \dual{\Theta \uu}{\vv}_V = b(\uu,\vv)\quad\forall \vv\in V
\end{align*}
and $\dual{\cdot}{\cdot}_V$ is the inner product in $V$ which induces the norm $\norm{\cdot}{V}$.
The following result is a consequence of the Babu\v{s}ka-Brezzi theory~\cite{b70,b74,xz03}.
For a proof of the best approximation property see~\cite[Thm.~2.2]{DemkowiczG_11_CDP}.
\begin{lemma}\label{lem:dpg}
  Define the operator $B:U\rightarrow V'$ by $B: \uu\mapsto(\vv\mapsto b(\uu,\vv))$
  and suppose that the assumptions from Lemma~\ref{lem:basic} hold. Then,~\eqref{eq:var} and~\eqref{eq:dpg}
  have unique solutions $\uu\in U$ and $\uu_\hp\in U_\hp$, respectively. Furthermore, they satisfy
  \begin{align}\label{lem:dpg:stab}
    \norm{\uu}{U,\opt} \leq \norm{\ell}{V'}
  \end{align}
  and
  \begin{align}\label{lem:dpg:opt}
    \norm{\uu-\uu_\hp}{U,\opt} = \inf_{\uu'_\hp\in U_\hp}\norm{\uu-\uu'_\hp}{U,\opt}.
  \end{align}
\end{lemma}
\subsection{Sobolev spaces}
We use the standard Sobolev spaces $L_2(\omega)$, $H^1(\omega)$, $H^1_0(\omega)$, $\HH(\div,\omega)$,
$\HH_0(\div,\omega)$
for Lipschitz domains $\omega\subset\R^d$. Vector-valued spaces and their elements will be denoted
by bold symbols.
In addition, we use spaces on the boundaries of Lipschitz domains $\omega$. Denoting by $\trace_\omega$ the
trace operator, we define
\begin{align*}
  H^{1/2}(\partial\omega) := \left\{ \trace_\omega u : u\in H^1(\omega) \right\}
  \quad\text{ and its dual }\quad H^{-1/2}(\partial\omega) := [H^{1/2}(\partial\omega)]'
\end{align*}
and equip them with the canonical trace norm and dual norm, respectively.
Here, duality is understood with respect to the extended $L_2$ inner product
$\dual{\cdot}{\cdot}_{\partial\omega}$.
The $L_2(\Omega)$ inner product will be denoted by $\vdual{\cdot}{\cdot}_\Omega$.
Let $\TT$ denote a disjoint partition of $\Omega$ into open Lipschitz sets $\el\in\TT$, i.e.,
$\cup_{\el\in\TT}\overline\el= \overline\Omega$. The set of all boundaries of all elements is the
skeleton $\cS := \left\{ \partial\el \mid \el\in\TT \right\}$.
By $\n_M$ we mean the outer normal vector on $\partial M$ for a Lipschitz set $M$.
On a partition, we use product spaces $H^1(\TT)$ and $\HH(\div,\TT)$, equipped with respective
product norms.
The symbols $\pwnabla$ and $\pwdiv$ denote the $\TT$-piecewise gradient and divergence operators.
We use spaces on the skeleton $\cS$ of $\TT$, namely
\begin{align*}
  H^{1/2}(\cS) &:=
  \Big\{ \wat u \in \Pi_{\el\in\TT}H^{1/2}(\partial\el):
  \exists w\in H^1(\Omega) \text{ such that } 
  \wat u|_{\partial\el} = w|_{\partial\el}\; \forall \el\in\TT \Big\},\\
  H^{-1/2}(\cS) &:= \Big\{
    \wat\sigma \in \Pi_{\el\in\TT}H^{-1/2}(\partial\el):
  \exists \qq\in\HH(\div,\Omega) \text{ such that } 
  \wat\sigma|_{\partial\el} = \qq\cdot\n_{\el}\;
  \forall\el\in\TT \Big\}.
\end{align*}
These spaces are equipped with the norms
\begin{align*}
  \norm{\wat u}{H^{1/2}(\cS)} &:= \inf
  \left\{ \norm{w}{H^1(\Omega)} : w\in H^1(\Omega) \text{ such that }
  \wat u|_{\partial\el}=w|_{\partial\el}\; \forall\el\in\TT \right\},\\
  \norm{\wat\sigma}{H^{-1/2}(\cS)} &:= \inf
  \left\{ \norm{\qq}{\HH(\div,\Omega)} : \qq\in\HH(\div,\Omega) \text{ such that }
  \wat\sigma|_{\partial\el}=\qq\cdot\n|_{\el}\; \forall\el\in\TT \right\}.
\end{align*}
Note that we think of the skeleton $\cS$ not as one geometric object, but rather as the set of
boundaries of all elements.
Consequently, we have defined $H^{1/2}(\cS)$ and $H^{-1/2}(\cS)$ not as canonical
trace spaces but as product spaces of trace spaces. This subtle difference
simplifies the subsequent analysis.
For two functions $\wat u \in \prod_{\el\in\TT}H^{1/2}(\partial\el)$ and
$\wat\sigma\in \prod_{\el\in\TT}H^{-1/2}(\partial\el)$ we use the notation
\begin{align*}
  \dual{\wat u}{\wat\sigma}_\cS
  := \sum_{\el\in\TT}\dual{\wat u|_\el}{\wat\sigma|_\el}_{\partial\el}.
\end{align*}
Note that for $v\in H^1(\Omega)$ and $\wat\sigma\in H^{-1/2}(\cS)$ integration by parts shows that
\begin{align*}
  \dual{v}{\wat\sigma}_\cS = \dual{v}{\wat\sigma}_\Gamma,
\end{align*}
and so the above left-hand side makes sense also for functions
$v\in H^1(\Omega)$ and $\wat\sigma\in H^{-1/2}(\Gamma)$.
We use an analogous duality pairing $\dual{\wat u}{\ttau\cdot\n}_\cS$
for functions $\wat u\in H^{1/2}(\Gamma)$ and $\ttau\in \HH(\div,\Omega)$.
For $v\in H^1(\TT)$ and $\ttau\in\HH(\div,\TT)$ we define norms of their jumps across $\cS$ by duality,
\begin{align*}
  \norm{\jump{v\n}}{1/2,\cS} := \sup_{\wat\sigma\in H^{-1/2}(\cS)}
  \frac{\dual{v}{\wat\sigma}_\cS}{\norm{\wat\sigma}{H^{-1/2}(\cS)}}
  \qquad \text{ and }\qquad 
  \norm{\jump{\ttau\cdot\n}}{1/2,\cS} := \sup_{\wat u\in H^{1/2}(\cS)}
  \frac{\dual{\wat u}{\ttau\cdot\n}_\cS}{\norm{\wat u}{H^{1/2}(\cS)}}.
\end{align*}
Here, $(\ttau\cdot\n)|_\el := \ttau|_\el\cdot\n_{\el}$ on $\partial\el$.
We will need the following estimates.
\begin{lemma}\label{lem:norms}
  There is a constant $C(\TT)>0$ which only depends on $\TT$ such that
  \begin{align}
    \norm{\jump{v\n}}{1/2,\cS}
    &\leq \norm{v}{H^1(\TT)} \quad\text{ for all }
    v\in H^1(\TT),\label{eq:jump:v}\\
    \norm{\jump{\ttau\cdot\n}}{1/2,\cS}
    &\leq \norm{\ttau}{\HH(\div,\TT)} \quad\text{ for all }
    \ttau\in \HH(\div,\TT)\label{eq:jump:tau}.
  \end{align}
  Furthermore, there holds
  \begin{align}
    \norm{\wat u}{H^{1/2}(\Gamma)}
    &\leq \norm{\wat u}{H^{1/2}(\cS)} \quad\text{ for all }
    \wat u\in H^{1/2}(\cS),\label{eq:trace:u}\\
    \norm{\wat\sigma}{H^{-1/2}(\Gamma)}
    &\leq \norm{\wat\sigma}{H^{1/2}(\cS)} \quad\text{ for all }
    \wat\sigma\in H^{-1/2}(\cS).\label{eq:trace:sigma}
  \end{align}
\end{lemma}
\begin{proof}
  The estimates~\eqref{eq:jump:v}--\eqref{eq:jump:tau} follow straightforwardly by integration
  by parts, e.g., cf.~\cite[Section~4.4]{dg11} for~\eqref{eq:jump:tau}.
  The estimate~\eqref{eq:trace:u}
  follows by definition of the norms in $H^{1/2}(\Gamma)$ and $H^{1/2}(\cS)$,
  \begin{align*}
    \norm{\wat u}{H^{1/2}(\Gamma)} = \inf \big\{ \norm{w}{H^1(\Gamma)} : w\in H^1(\Omega)
    \text{ such that } \wat u = \gamma_{\partial\Omega}w \big\} \leq \norm{\wat u}{H^{1/2}(\cS)}.
  \end{align*}
  The estimate~\eqref{eq:trace:sigma} follows the same way,
  using that $H^{-1/2}(\Gamma)$ is equivalently described as the space of normal components on $\Gamma$ of
  functions in $\HH(\div,\Omega)$, and that
  $\norm{\qq\cdot\n_\Omega}{H^{-1/2}(\Gamma)}\leq\norm{\qq}{\HH(\div,\Omega)}$,
  cf.~\cite[Cor.~2.8]{GiraultR_86_NS}.
\end{proof}
\subsection{Boundary integral operators}
In order to incorporate the PDE given in the exterior domain $\Omega^c$, the classical boundary
integral operators will be used. The fundamental solution
\begin{align*}
  G(z) :=
    \begin{cases}
    -\frac{1}{2\pi}\log\abs{z}\quad&\text{ for } d=2,\\
    \frac{1}{4\pi}\frac{1}{\abs{z}}\quad&\text{ for } d>2,
  \end{cases}
\end{align*}
of the Laplacian gives rise to the two potential operators $\slp$ and $\dlp$ defined by
\begin{align*}
  \slp\phi(x) := \int_\Gamma G(x-y)\phi(y)\,ds_y,\quad \text{ and }\quad
  \dlp v(x) := \int_\Gamma \partial_{\n_\Omega(y)}G(x-y)v(y)\,ds_y
  \quad\text{ for } x\in \R^d\setminus\Gamma.
\end{align*}
Then, boundary integral operators are defined as
$\slo := \trace_\Omega \slp$ (single layer operator) and $\dlo := 1/2 + \trace_\Omega\dlp$
(double layer operator) with adjoint $\adlo$. The operators $\slp: H^{-1/2}(\Gamma)\rightarrow H^1(\Omega)$,
$\slo: H^{-1/2}(\Gamma)\rightarrow H^{1/2}(\Gamma)$,
$\dlo: H^{1/2}(\Gamma)\rightarrow H^{1/2}(\Gamma)$, and $\adlo:H^{-1/2}(\Gamma)\rightarrow H^{-1/2}(\Gamma)$
are bounded, and there holds the representation formula
\begin{align}\label{eq:representationformula}
  \slo( \partial_{\n_\Omega} u^c ) + (1/2-\dlo)(u^c|_\Gamma) = 0
\end{align}
for solutions of the exterior PDE~\eqref{eq:tp:b}.
We refer to~\cite{Costabel_88_SIAM,HsiaoW_08_BIE,McLean_00_ES,Nedelec_82_IEOT}
for proofs and more details regarding boundary integral equations and the above operators. 
\subsection{Nonsymmetric coupling with ultra-weak formulation and main results}\label{section:results}
The trial space of our variational formulations will be
$U(\TT):=\LL_2(\Omega)\times L_2(\Omega)\times H^{1/2}(\cS)\times H^{-1/2}(\cS)$.
This space is a Hilbert space with norm
\begin{align*}
  \norm{(\ssigma,u,\wat u,\wat\sigma)}{U(\TT)}^2 :=
  \norm{\ssigma}{\LL_2(\Omega)}^2 + \norm{u}{L_2(\Omega)}^2 + \norm{\wat u}{H^{1/2}(\cS)}^2
  + \norm{\wat\sigma}{H^{-1/2}(\cS)}^2.
\end{align*}
In addition, we will need the space
$U(\Omega):=\LL_2(\Omega)\times L_2(\Omega)\times H^{1/2}(\Gamma)\times H^{-1/2}(\Gamma)$,
which is a Hilbert space with norm
\begin{align*}
  \norm{(\ssigma,u,\wat u,\wat\sigma)}{U(\Omega)}^2 &:=
  \norm{\ssigma}{\LL_2(\Omega)}^2 + \norm{u}{L_2(\Omega)}^2 + \norm{\wat u}{H^{1/2}(\Gamma)}^2
  + \norm{\wat\sigma}{H^{-1/2}(\Gamma)}^2.
\end{align*}
Note that the canonical restrictions of $\wat u$ and $\wat\sigma$ show that
$\uu\in U(\TT)$ can be viewed as an element of $U(\Omega)$. Using this restriction, we can regard
$U(\TT)$ as a subspace of $U(\Omega)$. However, $\norm{\cdot}{U(\Omega)}$ is only a seminorm on $U(\TT)$.
The test space of our formulation will be $V(\TT) := H^1(\TT)\times \HH(\div,\TT)\times H^{-1/2}(\Gamma)$,
being a Hilbert space with norm
\begin{align*}
  \norm{\vv}{V(\TT)}^2 := \norm{v}{H^1(\TT)}^2 + \norm{\ttau}{\HH(\div,\TT)}^2
  + \norm{\psi}{H^{-1/2}(\Gamma)}^2.
\end{align*}
In addition, we will need the space $V(\Omega) := H^1(\Omega)\times \HH(\div,\Omega)\times H^{-1/2}(\Gamma)$,
which is a Hilbert space with norm
\begin{align*}
  \norm{\vv}{V(\Omega)}^2 := \norm{v}{H^1(\Omega)}^2 + \norm{\ttau}{\HH(\div,\Omega)}^2
  + \norm{\psi}{H^{-1/2}(\Gamma)}^2.
\end{align*}
Note that $V(\Omega)\subset V(\TT)$.
The variational formulation that we will analyze is the following Johnson-N\'ed\'elec type coupling:
find $\uu := (\ssigma,u,\wat u,\wat\sigma)$ such that
\begin{subequations}\label{eq:dfb}
\begin{align}
  \vdual{\ssigma}{\pwnabla v}_\Omega - \dual{\wat\sigma}{\jump{v}}_\cS
   &= \vdual{f}{v}_\Omega\label{eq:dfb:1}\\
   \vdual{\ssigma}{\ttau}_\Omega + \vdual{u}{\pwdiv\ttau}_\Omega - \dual{\wat u}{\jump{\ttau\cdot\n}}_\cS
   &= 0\label{eq:dfb:2}\\
  \dual{\slo \wat\sigma}{\psi}_\Gamma + \dual{(1/2-\dlo)\wat u}{\psi}_\Gamma 
  &= \dual{(1/2-\dlo)u_0+\slo\phi_0}{\psi}_\Gamma\label{eq:dfb:3}
\end{align}
\end{subequations}
for appropriate test functions $\vv:=(v,\ttau,\psi)$.
The equations~\eqref{eq:dfb:1}--~\eqref{eq:dfb:2}
are obtained by treating the interior PDE~\eqref{eq:tp:a} as in
the DPG-finite element method, cf.~\cite{dg11}, i.e., writing it as a first order system, testing
with appropriate functions, integrating by parts piecewise, and replacing the appearing boundary terms
by new unknowns $\wat u$ and $\wat\sigma$. These new unknowns
already involve the interior trace and normal derivative of $u$ on $\Gamma$,
which are coupled to the exterior problem by using the interface
conditions~\eqref{eq:tp:c}--\eqref{eq:tp:d} in the representation formula~\eqref{eq:representationformula}.
In contrast to that, in the classical nonsymmetric coupling, cf.~\eqref{eq:jn} below, the unknowns are
$u\in H^1(\Omega)$ and $\phi\in H^{-1/2}(\Gamma)$, where $\phi$ is the normal derivative of $u^c$ on $\Gamma$.

The bilinear form on the left-hand side of~\eqref{eq:dfb} will be called $b(\uu,\vv)$, and
the linear form on the right-hand side will be called $\ell(\vv)$. We will use two different formulations
which differ in the underlying spaces; the one we actually analyze and solve numerically is
\begin{align}\label{eq:b:1}
  \text{find } \uu\in U(\TT)\text{ such that } b(\uu,\vv)=\ell(\vv) \text{ for all } \vv\in V(\TT).
\end{align}
The second one is only of theoretical interest and will be needed in the proofs of the main
theorems, it is
\begin{align}\label{eq:b:2}
  \text{find } \uu\in U(\Omega)\text{ such that } b(\uu,\vv)=\ell(\vv) \text{ for all } \vv\in V(\Omega).
\end{align}
The first result of this work states unique solvability and stability
of the variational formulation~\eqref{eq:b:1}. The proof will be given in Section~\ref{sec:proofs} below.
\begin{theorem}\label{thm:solve}
  For given $f\in L_2(\Omega)$, $u_0\in H^{1/2}(\Gamma)$, and $\phi_0\in H^{-1/2}(\Gamma)$ and given
  partition $\TT$ of $\Omega$, the variational formulation~\eqref{eq:b:1}
  has a unique solution $(\ssigma,u,\wat u,\wat\sigma)\in U(\TT)$. Furthermore,
  \begin{align*}
    \norm{(\ssigma,u,\wat u,\wat\sigma)}{U(\Omega)}
    \lesssim \norm{f}{L_2(\Omega)} + \norm{u_0}{H^{1/2}(\Gamma)}
    + \norm{\phi_0}{H^{-1/2}(\Gamma)}.
  \end{align*}
  The hidden constant in $\lesssim$ only depends on $\Omega$.
\end{theorem}
The second main result of this work is the following quasi optimality result of the Petrov-Galerkin method
with optimal test functions associated with the norm $\norm{\cdot}{V(\TT)}$.
\begin{theorem}\label{thm:opt}
  Suppose that $U_\hp(\TT)\subset U(\TT)$ is a discrete subspace.
  Then, the discrete formulation~\eqref{eq:dpg}, where $V:=(V(\TT),\norm{\cdot}{V(\TT)})$,
  has a unique solution $(\ssigma_\hp,u_\hp,\wat u_\hp,\wat\sigma_\hp)\in U_\hp(\TT)$.
  Furthermore,
  \begin{align*}
    \norm{(\ssigma-\ssigma_\hp,u-u_\hp,&\wat u-\wat u_\hp,\wat\sigma-\wat\sigma_\hp)}{U(\Omega)}
    \lesssim\\
    &\inf_{(\ssigma_\hp',u_\hp',\wat u_\hp',\wat\sigma_\hp')\in U_\hp(\TT)}
    \norm{(\ssigma-\ssigma_\hp',u-u_\hp',\wat u-\wat u_\hp',\wat\sigma-\wat\sigma_\hp')}{U(\TT)},
  \end{align*}
  where $(\ssigma,u,\wat u,\wat\sigma)\in U(\TT)$ is the exact solution of~\eqref{eq:b:1}.
  The hidden constant in $\lesssim$ only depends on $\Omega$.
\end{theorem}
\begin{remark}
  Note that the norm $\norm{\cdot}{U(\Omega)}$
  in the stability estimate and on the left-hand side of the quasi-optimality
  is a weaker norm in $U(\TT)$.
  This norm does not control the parts of $\wat u$ and $\wat \sigma$ on the inner parts
  of the skeleton $\cS$, i.e., the parts which are not on the boundary $\Gamma$.
  However, we have full control of 
  the Cauchy data $\wat u$ and $\wat \sigma$ on the boundary $\Gamma$,
  which are the only ingredients to solve the exterior problem.
\end{remark}
For the proofs of Theorems~\ref{thm:solve} and~\ref{thm:opt} we will apply the results of
Lemmas~\ref{lem:basic} and~\ref{lem:dpg}, hence we have to check the assumptions of
Lemma~\ref{lem:basic}. The boundedness of the bilinear form is shown in Lemma~\ref{lem:stab}.
The bijectivity of the operator that corresponds to the bilinear form will be
proved in Section~\ref{sec:proofs} below.
In a first step, this will yield the results of Lemma~\ref{lem:dpg}, i.e., stability
and quasi-optimality in the norm $\norm{\cdot}{U(\TT),\opt}$ defined in Lemma~\ref{lem:basic}.
To obtain the results in the main theorems, it remains to relate the
norm $\norm{\cdot}{U(\TT),\opt}$
to the norms $\norm{\cdot}{U(\TT)}$ and $\norm{\cdot}{U(\Omega)}$.
This will be done by characterizing the optimal test norms 
$\norm{\cdot}{V(\TT),\opt}$ (optimal to $\norm{\cdot}{U(\TT)}$)
and 
$\norm{\cdot}{V(\TT),\alpha}$ (optimal to $\norm{\cdot}{U(\Omega)}$)
and by relating them to the norm $\norm{\cdot}{V(\TT)}$ (optimal to $\norm{\cdot}{U(\TT),\opt}$).
These norm equivalences are the topic of Section~\ref{section:norms}.
\section{Technical results}\label{section:adjoint}
We start by showing the boundedness of the bilinear form $b$.
\begin{lemma}\label{lem:stab}
  It holds that
  \begin{align}
    b(\uu,\vv) &\lesssim
    \norm{\uu}{U(\TT)}\norm{\vv}{V(\TT)} \quad \text{ for all } \uu\in U(\TT), \vv\in V(\TT),\label{eq:stab:1}\\
    b(\uu,\vv) &\lesssim
    \norm{\uu}{U(\Omega)}\norm{\vv}{V(\Omega)} \quad \text{ for all } \uu\in U(\Omega), \vv\in V(\Omega).
    \label{eq:stab:2}
  \end{align}
  The hidden constant only depends on $\Omega$.
\end{lemma}
\begin{proof}
  The proof of~\eqref{eq:stab:1} follows immediately by the Cauchy-Schwarz inequality
  and standard boundedness properties. Note that by definition of the norms
  $\norm{\jump{(\cdot)\n}}{1/2,\cS}$ and $\norm{\jump{(\cdot) \cdot \n}}{1/2,\cS}$
  and~\eqref{eq:jump:v} and \eqref{eq:jump:tau}, it holds that
  \begin{align*}
    \abs{\dual{\wat\sigma}{v}_\cS} &\leq \norm{\wat\sigma}{H^{-1/2}(\Gamma)} \; \norm{\jump{v\n}}{1/2,\cS}
    \leq \norm{\wat\sigma}{H^{-1/2}(\Gamma)} \; \norm{v}{H^1(\TT)},\\
    \abs{\dual{\wat u}{\ttau\cdot\n}} &\leq \norm{\jump{\ttau\cdot\n}}{1/2,\cS} \; \norm{\wat u}{H^{1/2}(\cS)}
    \leq \norm{\ttau}{\HH(\div,\TT)} \; \norm{\wat u}{H^{1/2}(\cS)}.
  \end{align*}
  Furthermore, the boundedness of the operators $\slo$, $\dlo$ and~\eqref{eq:trace:u},~\eqref{eq:trace:sigma}
  yield
  \begin{align*}
    \abs{\dual{\slo\wat\sigma}{\psi}_\Gamma} &\lesssim
    \norm{\wat\sigma}{H^{-1/2}(\Gamma)} \; \norm{\psi}{H^{-1/2}(\Gamma)}
    \leq \norm{\wat\sigma}{H^{-1/2}(\cS)} \; \norm{\psi}{H^{-1/2}(\Gamma)}\\
    \abs{\dual{(1/2-\dlo)\wat u}{\psi}_\Gamma} &\lesssim
    \norm{\wat u}{H^{1/2}(\Gamma)} \; \norm{\psi}{H^{-1/2}(\Gamma)}
    \leq \norm{\wat u}{H^{1/2}(\cS)} \; \norm{\psi}{H^{-1/2}(\Gamma)}.
  \end{align*}
  For the proof of~\eqref{eq:stab:2}, note in addition that
  \begin{align*}
    \abs{\dual{\wat\sigma}{v}_\cS} \leq
    \abs{\dual{\wat\sigma}{v}_\Gamma}
    \leq \norm{\wat\sigma}{H^{-1/2}(\Gamma)} \; \norm{v}{H^{1/2}(\Gamma)}
    \leq \norm{\wat\sigma}{H^{-1/2}(\Gamma)} \; \norm{v}{H^{1}(\Omega)}
  \end{align*}
  due to the definition of the norm $\norm{\cdot}{H^{1/2}(\Gamma)}$.
  The part $\abs{\dual{\wat u}{\ttau\cdot\n}}$ is treated the same
  way.
\end{proof}
\subsection{Johnson-N\'ed\'elec coupling}\label{section:surj}
The aim of this subsection is to show that our new formulation is equivalent to 
the classical nonsymmetric coupling.
As we will see in Section~\ref{sec:proofs}, this implies, in particular, injectivity
of the operator $B:U\to V'$ that corresponds to the bilinear form $b(\cdot,\cdot)$.

The transmission problem~\eqref{eq:tp} can be written equivalently as:
Given $(\GG,F,\rho,\lambda)\in \LL_2(\Omega)\times L_2(\Omega)
\times H^{-1/2}(\Gamma) \times H^{1/2}(\Gamma)$,
find $(u,\phi)\in H^1(\Omega)\times H^{-1/2}(\Gamma)$ such that
\begin{subequations}\label{eq:jn}
\begin{align}
  \vdual{\nabla u}{\nabla v}_\Omega - \dual{\phi}{v}_\Gamma
  &= \vdual{\GG}{\nabla v}_\Omega + \vdual{F}{v}_\Omega + \dual{\rho}{v}_\Gamma\label{eq:jn:1}\\
  \dual{(1/2-\dlo)u}{\psi}_\Gamma + \dual{\slo\phi}{\psi}_\Gamma &= \dual{\psi}{\lambda}_\Gamma\label{eq:jn:2}
\end{align}
\end{subequations}
for all $(v,\psi)\in H^1(\Omega)\times H^{-1/2}(\Gamma)$. 
Proof of unique solvability is not straightforward
as Problem~\eqref{eq:jn} is not elliptic, and was addressed recently
in~\cite{s13} and also in~\cite{affkmp13,OfS_14_FEMBEM,Steinbach_11_FEMBEM}. Following
the approach of~\cite{affkmp13}, the following stability result can be shown.
\begin{lemma}\label{lem:jn}
  For given $(\GG,F,\rho,\lambda)\in
    \LL_2(\Omega)\times L_2(\Omega) \times H^{-1/2}(\Gamma) \times H^{1/2}(\Gamma)$,
  the variational formulation~\eqref{eq:jn} has a unique solution
  $(u,\phi)\in H^1(\Omega)\times H^{-1/2}(\Gamma)$, and
  \begin{align*}
    \norm{u}{H^1(\Omega)} + \norm{\phi}{H^{-1/2}(\Gamma)} \lesssim \norm{\GG}{\LL_2(\Omega)} + 
    \norm{F}{L_2(\Omega)} + \norm{\rho}{H^{-1/2}(\Gamma)} + \norm{\lambda}{H^{1/2}(\Gamma)}.
  \end{align*}
\end{lemma}
\begin{proof}
  Denote the bilinear form on the left-hand side of~\eqref{eq:jn} as $b_\jn\left( u,\phi; v,\psi \right)$
  and the linear functional on the right-hand side as $\ell_\jn(v,\psi)$. Consider the problem of
  finding $(u,\phi)\in H^1(\Omega)\times H^{-1/2}(\Gamma)$ such that 
  \begin{align}
    \label{eq:jn:stab}
    \begin{split}
      b_\jn\left( u,\phi; v,\psi \right) +
      \dual{1}{(1/2-\dlo)u+\slo\phi}_\Gamma&\dual{1}{(1/2-\dlo)v+\slo\psi}_\Gamma\\
      &= \ell_\jn(v,\psi) + \dual{1}{\lambda}_\Gamma\dual{1}{(1/2-\dlo)v+\slo\psi}_\Gamma
    \end{split}
  \end{align}
  for all $(v,\psi)\in H^1(\Omega)\times H^{-1/2}(\Gamma)$. According to~\cite[Thm.~14]{affkmp13},
  a solution $(u,\phi)\in H^1(\Omega)\times H^{-1/2}(\Gamma)$ of~\eqref{eq:jn} also solves~\eqref{eq:jn:stab}
  and vice versa. Furthermore,~\cite[Thm.~15]{affkmp13} states that the bilinear form on the left-hand side
  of~\eqref{eq:jn:stab} is continuous and elliptic on $H^1(\Omega)\times H^{-1/2}(\Gamma)$. The
  norm of the linear functional on the right-hand side of~\eqref{eq:jn:stab} is bounded by
  \begin{align*}
    \norm{\GG}{\LL_2(\Omega)} + \norm{F}{L_2(\Omega)} + \norm{\rho}{H^{-1/2}(\Gamma)}
    + \norm{\lambda}{H^{1/2}(\Gamma)}.
  \end{align*}
  We finish the proof by application of the Lax-Milgram lemma.
\end{proof}
The next lemma shows that our new formulation~\eqref{eq:b:1} is equivalent to the classical
formulation~\eqref{eq:jn}.
\begin{lemma}\label{lem:equi}
  Let $f\in L_2(\Omega)$, $u_0\in H^{1/2}(\Gamma)$, and $\phi_0\in H^{-1/2}(\Gamma)$.
  Define $\GG=0$, $F=f$, $\rho=\phi_0$, and $\lambda=(1/2-\dlo)u_0$. Then there hold the following
  statements:
  \begin{itemize}
    \item[(i)] Suppose that $(u,\phi)\in H^1(\Omega)\times H^{-1/2}(\Gamma)$ is a solution of~\eqref{eq:jn}.
      Then $\nabla u\in \HH(\div,\Omega)$,
      and $(\nabla u, u, \trace u, \nabla u\cdot \n)\in U(\TT)$ is a solution of~\eqref{eq:b:1}
      and of~\eqref{eq:b:2}.
    \item[(ii)] Suppose that $(\ssigma, u,\wat u,\wat\sigma)\in U(\TT)$ respectively
      $(\ssigma, u,\wat u,\wat\sigma)\in U(\Omega)$ is a solution
      of~\eqref{eq:b:1}, respectively~\eqref{eq:b:2}.
      Then $(u,\wat\sigma|_\Gamma-\phi_0)\in H^1(\Omega)\times H^{-1/2}(\Gamma)$
      is a solution of~\eqref{eq:jn} and $\ssigma=\nabla u$ as well as $\wat u=u$,
      $\wat\sigma = \nabla u \cdot\n$ on $\cS$, respectively on $\Gamma$.
  \end{itemize}
\end{lemma}
\begin{proof}
  We first show \textit{(i)}.
  Using $v\in C_0^\infty(\Omega)$ in~\eqref{eq:jn:1} shows that
  \begin{align}\label{lem:equi:100}
    \div\nabla u=-f\in L_2(\Omega).
  \end{align}
  Hence, $\nabla u\in \HH(\div,\Omega)$ and eventually $\nabla u\cdot\n \in H^{-1/2}(\cS)$.
  This yields $(\nabla u, u, \trace u, \nabla u\cdot \n)\in U(\TT)$.
  Identity~\eqref{lem:equi:100} and integration by parts implies~\eqref{eq:dfb:1}
  for all $v\in H^1(\TT)$. Integration by parts also shows~\eqref{eq:dfb:2}
  for all $\ttau\in \HH(\div,\TT)$. From~\eqref{lem:equi:100} and~\eqref{eq:jn:1}
  we conclude that
  \begin{align*}
    \dual{\phi}{t}_\Gamma = \dual{\nabla u\cdot\n}{t}_\Gamma - \dual{\phi_0}{t}_\Gamma \quad
    \text{ for all } t\in H^{1/2}(\Gamma).
  \end{align*}
  Using the symmetry of $\slo$, this leads us to
  \begin{align*}
    \dual{\slo\phi}{\psi}_\Gamma = \dual{\slo\nabla u \cdot\n}{\psi}_\Gamma - \dual{\slo\phi_0}{\psi}_\Gamma
    \quad\text{ for all } \psi\in H^{-1/2}(\Gamma)
  \end{align*}
  If we plug the last identity into~\eqref{eq:jn:2}, we obtain exactly~\eqref{eq:dfb:3}
  for all $\psi\in H^{-1/2}(\Gamma)$. In total, $(\nabla u, u, \trace u, \nabla u\cdot \n)\in U(\TT)$
  is a solution of~\eqref{eq:b:1}. Furthermore, it is also a solution of~\eqref{eq:b:2}. This
  follows immediately as $(\nabla u, u, \trace u, \nabla u\cdot \n)\in U(\Omega)$
  by the canonical restriction, and as $V(\Omega)\subset V(\TT)$.
  
  Now we show \textit{(ii)}. To that end, denote by
  $(\ssigma,u,\wat u,\wat\sigma)\in U(\TT)$ a solution of~\eqref{eq:b:1}. 
  Equation~\eqref{eq:dfb:2} first shows that $\ssigma = \nabla u$ and hence $u\in H^1(\Omega)$, as well
  as $\widehat u = u|_\cS$.
  As $\wat\sigma\in H^{-1/2}(\cS)$, we have $\wat\sigma|_\Gamma-\phi_0\in H^{-1/2}(\Gamma)$,
  and from~\eqref{eq:dfb:1} follows~\eqref{eq:jn:1}. Likewise,~\eqref{eq:jn:2} follows
  from~\eqref{eq:dfb:3} using $\widehat u = u|_\cS$.
  The case that $(\ssigma,u,\wat u,\wat\sigma)\in U(\TT)$ is a solution of~\eqref{eq:b:2} follows
analogously.
\end{proof}
We additionally need the following stronger result, which shows the surjectivity of the operator associated
to our bilinear form.
\begin{lemma}\label{lem:surj}
  For every $\ell \in V(\TT)'$ there exists $\uu\in U(\TT)$ with $b(\uu,\vv) = \ell(\vv)$ for all $\vv\in V(\TT)$.
\end{lemma}
\begin{proof}
  Using the Riesz representation theorem, we write
  \begin{align*}
    \ell(\vv) = \vdual{S}{v}_\Omega + \vdual{\sS}{\pwnabla v}_\Omega
    + \vdual{\tT}{\ttau}_\Omega + \vdual{T}{\pwdiv\ttau}_\Omega
    + \dual{\mu}{\psi}_\Gamma
  \end{align*}
  with $\sS = \pwnabla S$ and $T = \pwdiv\tT$ and $\mu\in H^{1/2}(\Gamma)$.
  According to Lemma~\ref{lem:jn}, there is a unique solution
  $(\wilde u,\wilde\phi)\in H^1(\Omega)\times H^{-1/2}(\Omega)$ of~\eqref{eq:jn} with right-hand side data
  $F:=S\in L_2(\Omega)$, $\GG := \sS-\tT\in\LL_2(\Omega)$, arbitrary $\lambda\in H^{1/2}(\Gamma)$
  and $\rho := \slo^{-1}(\mu-\lambda)\in H^{-1/2}(\Gamma)$. From~\eqref{eq:jn:1} it follows
  that $\nabla \wilde u-\sS+\tT\in H(\div,\Omega)$ with $-\div(\nabla \wilde u-\sS+\tT) = S$.
  Now define $\ssigma := \nabla\wilde u + \tT$, $u := \wilde u + T\in L_2(\Omega)$,
  $\wat u = \wilde u|_\cS$, and $\wat\sigma := (\nabla \wilde u-\sS+\tT)\cdot\n \in H^{-1/2}(\cS)$.
  Integration by parts shows
  \begin{align*}
    \vdual{\ssigma}{\pwnabla v}_\Omega - \dual{\wat\sigma}{\jump{\ttau\cdot\n}}_\cS
    &= \vdual{S}{v}_\Omega + \vdual{\sS}{\pwnabla v}_\Omega,\\
    \vdual{\ssigma}{\ttau}_\Omega + \vdual{u}{\pwdiv\ttau}_\Omega - 
    \dual{\wat u}{\jump{v}}_\cS &= \vdual{\tT}{\ttau}_\Omega + \vdual{T}{\pwdiv\ttau}_\Omega.
  \end{align*}
  Furthermore,~\eqref{eq:jn:1} shows that $\wat\sigma|_\Gamma = \rho+\phi$.
  Therefore, by definition of $\rho$ and~\eqref{eq:jn:2}, we conclude that
  $\slo\wat\sigma = \mu - (1/2-\dlo)\wat u$.
  We have thus shown that $b(\uu,\vv) = \ell(\vv)$ for all $\vv\in V(\TT)$.
\end{proof}
\subsection{Bielak-MacCamy coupling and norm equivalences}\label{section:norms}
In order to relate the norm $\norm{\cdot}{U(\TT),\opt}$ to a norm of our choice, we will investigate
norm equivalences in the test spaces. To that end define seminorms in $V(\Omega)$ and $V(\TT)$ by
\begin{align*}
  \norm{\vv}{V(\Omega)}
  &:= \norm{\ttau+\pwnabla v}{\LL_2(\Omega)} + \norm{\pwdiv\ttau}{L_2(\Omega)}\\
  &\qquad +\norm{(1/2-\adlo)\psi - \ttau\cdot\n}{H^{-1/2}(\Gamma)} + \norm{\slo\psi - v}{H^{1/2}(\Gamma)},\\
  \norm{\vv}{V(\TT),\opt} &:= \norm{\ttau+\pwnabla v}{\LL_2(\Omega)} + \norm{\pwdiv\ttau}{L_2(\Omega)}\\
  &\qquad +\norm{\jump{ (\ttau-\EE(1/2-\adlo)\psi)\cdot\n}}{1/2,\cS} + \norm{\jump{(v-\slp\psi)\n}}{1/2,\cS}.
\end{align*}
Here, $\EE:H^{-1/2}(\Gamma)\rightarrow \HH(\div,\Omega)$ is a bounded and linear extension operator, i.e.,
$(\EE\psi)\cdot\n_\Omega = \psi$. See~\cite[Cor.~2.8]{GiraultR_86_NS} for an explicit construction of $\EE$.
Equivalence of norms in the test space amounts to an analysis of the adjoint problem.
In case of the nonsymmetric coupling, the adjoint problem is the so-called Bielak-MacCamy coupling,
which first appeared in~\cite{bm83}. Given $(\GG,F,\rho,\lambda)\in \LL_2(\Omega)\times L_2(\Omega)
\times H^{-1/2}(\Gamma) \times H^{1/2}(\Gamma)$,
it consists in finding $(v,\psi)\in H^1(\Omega)\times H^{-1/2}(\Gamma)$ such that
\begin{subequations}\label{eq:bmc}
\begin{align}
  \vdual{\nabla v}{\nabla u}_\Omega + \dual{(1/2-\adlo)\psi}{u}_\Gamma &= 
  \vdual{\GG}{\nabla u}_\Omega + \vdual{F}{u}_\Omega + \dual{\rho}{u}_\Gamma\label{eq:bmc:1}\\
  \dual{\phi}{\slo\psi}_\Gamma - \dual{\phi}{v}_\Gamma &= \dual{\phi}{\lambda}_\Gamma\label{eq:bmc:2}
\end{align}
\end{subequations}
for all $(u,\phi)\in H^1(\Omega)\times H^{-1/2}(\Gamma)$. Again, proof of the existence of a unique solution
is not straightforward since the problem is not elliptic. However, using ideas from~\cite{affkmp13,ghs12},
the following result can be shown.
\begin{lemma}\label{lem:bmc}
  For given $(\GG,F,\rho,\lambda)\in
    \LL_2(\Omega)\times L_2(\Omega) \times H^{-1/2}(\Gamma) \times H^{1/2}(\Gamma)$,
  the variational formulation~\eqref{eq:bmc} has a unique solution
  $(v,\psi)\in H^1(\Omega)\times H^{-1/2}(\Gamma)$, and
  \begin{align*}
    \norm{v}{H^1(\Omega)} + \norm{\psi}{H^{-1/2}(\Gamma)} \lesssim \norm{\GG}{\LL_2(\Omega)} + 
    \norm{F}{L_2(\Omega)} + \norm{\rho}{H^{-1/2}(\Gamma)} + \norm{\lambda}{H^{1/2}(\Gamma)}.
  \end{align*}
\end{lemma}
\begin{proof}
  Denote the bilinear form on the left-hand side of~\eqref{eq:bmc} as $b_\bm\left( v,\psi; u,\phi \right)$
  and the linear functional on the right-hand side as $\ell_\bm(u,\phi)$. Consider the problem of
  finding $(v,\psi)\in H^1(\Omega)\times H^{-1/2}(\Gamma)$ such that 
  \begin{align}\label{eq:bmc:stab}
    b_\bm\left( v,\psi; u,\phi \right) + \dual{1}{\slo\psi-v}_\Gamma\dual{1}{\slo\phi-u}_\Gamma
    = \ell_\bm(u,\phi) + \dual{1}{\lambda}_\Gamma\dual{1}{\slo\phi-u}_\Gamma
  \end{align}
  for all $(u,\phi)\in H^1(\Omega)\times H^{-1/2}(\Gamma)$. According to~\cite[Thm.~8]{affkmp13},
  a solution $(v,\psi)\in H^1(\Omega)\times H^{-1/2}(\Gamma)$ of~\eqref{eq:bmc} also solves~\eqref{eq:bmc:stab}
  and vice versa. Furthermore,~\cite[Thm.~9]{affkmp13} states that the bilinear form on the left-hand side
  of~\eqref{eq:bmc:stab} is continuous and elliptic on $H^1(\Omega)\times H^{-1/2}(\Gamma)$. The
  norm of the linear functional on the right-hand side of~\eqref{eq:bmc:stab} is bounded by
  \begin{align*}
    \norm{\GG}{\LL_2(\Omega)} + \norm{F}{L_2(\Omega)} + \norm{\rho}{H^{-1/2}(\Gamma)}
    + \norm{\lambda}{H^{1/2}(\Gamma)}.
  \end{align*}
  Hence, by application of the Lax-Milgram lemma, the statement is proved.
\end{proof}
We need the following extension of~\cite[Lem.~4.4]{dg11}.
\begin{lemma}\label{lem:BMC:inhom}
  Let $\GG\in\LL_2(\Omega)$, $F\in L_2(\Omega)$, $\rho\in H^{-1/2}(\Gamma)$, and $\lambda\in H^{1/2}(\Gamma)$.
  Then, there exists a unique solution
  $(v_1,\ttau_1,\psi_1)\in H^1(\Omega)\times\HH(\div,\Omega)\times H^{-1/2}(\Gamma)$ satisfying
  \begin{subequations}\label{lem:BMC:inhom:eq}
  \begin{align}
    \ttau_1 + \nabla v_1 &= \GG\label{lem:BMC:inhom:eqA}\; \text{ in } \Omega,\\
    \div\ttau_1 &= F\label{lem:BMC:inhom:eqB}\; \text{ in } \Omega,\\
    (1/2-\adlo)\psi_1 - \ttau_1\cdot\n &= \rho\label{lem:BMC:inhom:eqC}\; \text{ on } \Gamma,\\
    \slo\psi_1 - v_1 &= \lambda\label{lem:BMC:inhom:eqD}\; \text{ on } \Gamma.
  \end{align}
  \end{subequations}
  Furthermore,
  \begin{align}\label{lem:BMC:inhom:stab}
    \norm{\ttau_1}{\HH(\div,\Omega)} + \norm{v_1}{H^1(\Omega)} + \norm{\psi_1}{H^{-1/2}(\Gamma)}
    \lesssim \norm{\GG}{\LL_2(\Omega)} + \norm{F}{L_2(\Omega)} + \norm{\rho}{H^{-1/2}(\Gamma)}
    +\norm{\lambda}{H^{1/2}(\Gamma)}.
  \end{align}
\end{lemma}
\begin{proof}
  Choose $(v_1,\psi_1)$ as solution of the Bielak-MacCamy coupling~\eqref{eq:bmc} with the respective
  data $\GG, F, \rho$, and $\lambda$.
  By Lemma~\ref{lem:bmc}, such a solution exists uniquely and fulfills~\eqref{lem:BMC:inhom:eqD} as well as
  \begin{align}\label{lem:BMC:inhom:eq:1}
    \norm{v_1}{H^1(\Omega)} + \norm{\psi_1}{H^{-1/2}(\Gamma)}
    \lesssim \norm{\GG}{\LL_2(\Omega)} + \norm{F}{L_2(\Omega)} + \norm{\rho}{H^{-1/2}(\Gamma)}
    +\norm{\lambda}{H^{1/2}(\Gamma)}.
  \end{align}
  Now define $\ttau_1 := \GG-\nabla v_1$, i.e.,~\eqref{lem:BMC:inhom:eqA} holds.
  At first, $\ttau_1\in \LL_2(\Omega)$ only, but testing~\eqref{eq:bmc:1}
  with $u\in H^1_0(\Omega)$ shows that $\ttau_1\in\HH(\div,\Omega)$ with $\div\ttau_1=F$,
  i.e.,~\eqref{lem:BMC:inhom:eqB}.
  Then, testing ~\eqref{eq:bmc:1} with $u\in H^1(\Omega)$ shows that
  \begin{align*}
    \vdual{\ttau_1}{\nabla u}_\Omega = \dual{(1/2-\adlo)\psi_1}{u}_\Gamma - \dual{\rho}{u}_\Gamma
    - \vdual{\div\ttau_1}{u}_\Omega,
  \end{align*}
  which gives~\eqref{lem:BMC:inhom:eqC}. By definition of $\ttau_1$, there holds
  \begin{align*}
    \norm{\ttau_1}{\HH(\div,\Omega)} \lesssim \norm{\GG}{\LL_2(\Omega)} + \norm{\nabla v_1}{\LL_2(\Omega)}
    + \norm{F}{L_2(\Omega)},
  \end{align*}
  and together with~\eqref{lem:BMC:inhom:eq:1}, this shows~\eqref{lem:BMC:inhom:stab}.
  
  To see that a solution of~\eqref{lem:BMC:inhom:eq} is unique, assume that
  $(\wilde v_1,\wilde \ttau_1, \wilde\psi_1)$ is a solution of~\eqref{lem:BMC:inhom:eq}
  with vanishing right-hand side. Then $-\Delta\wilde v_1=0$ in $\Omega$
  and $\slo\wilde v_1=\wilde\psi_1$ on $\Gamma$. Therefore, $\wilde v_1 = \slp\wilde\psi_1$ in $\Omega$
  and hence $\partial_{\n}\wilde v_1 = (1/2 + \adlo)\wilde\psi_1$. However,
  ~\eqref{lem:BMC:inhom:eqA} and~\eqref{lem:BMC:inhom:eqC}
  also show that $\partial_{\n}\wilde v_1 = -\wilde\ttau_1\cdot\n = -(1/2-\adlo)\wilde\psi_1$.
  We conclude that $\wilde\psi_1=0$, which implies $\wilde v_1=0$ and $\wilde\ttau_1=0$.
\end{proof}
\begin{lemma}\label{lem:adjoint:jn}
  For all $\ttau\in \HH(\div,\Omega)$, $v\in H^1(\Omega)$, and $\psi\in H^{-1/2}(\Gamma)$ there holds
  \begin{align*}
    \norm{v}{H^1(\Omega)}^2 + \norm{\ttau}{\HH(\div,\Omega)}^2 + \norm{\psi}{H^{-1/2}(\Gamma)}^2
    \lesssim \norm{(v,\ttau,\psi)}{V(\Omega)}^2.
  \end{align*}
  The constant hidden in $\lesssim$ only depends on $\Omega$.
\end{lemma}
\begin{proof}
  We set in Lemma~\ref{lem:BMC:inhom} the right-hand side to be 
  $\GG := \ttau + \pwnabla v$, $F:=\pwdiv\ttau$,
  $\rho := (1/2-\adlo)\psi - \ttau\cdot\n$, and $\lambda := \slo\psi-v$.
  Then, $(v,\ttau,\psi)$ is the unique solution of~\eqref{lem:BMC:inhom:eq},
  and the statement follows from~\eqref{lem:BMC:inhom:stab}.
\end{proof}
\begin{lemma}\label{lem:testnorms:equi:jn}
  For all $(\ttau,v,\psi)\in V(\TT)$ there holds
  \begin{align*}
    \norm{(v,\ttau,\psi)}{V(\TT),\opt} \lesssim \norm{(v,\ttau,\psi)}{V(\TT)} 
  \end{align*}
  where the hidden constant only depends on $\Omega$.
\end{lemma}
\begin{proof}
  The statement follows with the triangle inequality, the estimates~\eqref{eq:jump:v}--\eqref{eq:jump:tau},
  and the continuity of the operators $\slp: H^{-1/2}(\Gamma)\rightarrow H^1(\Omega)$,
  $\EE:H^{-1/2}(\Gamma)\rightarrow \HH(\div,\Omega)$, and
  $\adlo:H^{-1/2}(\Gamma)\rightarrow H^{-1/2}(\Gamma)$.
\end{proof}
\section{Proofs of main theorems}\label{sec:proofs}
The next lemma shows that our new bilinear form $b$ is definite.
\begin{lemma}\label{lem:definite}
  The bilinear form $b$, given by the left-hand side of~\eqref{eq:dfb},
  is definite on $U(\TT)\times V(\TT)$ as well as on $U(\Omega)\times V(\Omega)$, i.e.,
  \begin{align*}
    b(\uu,\vv) &= 0 \quad \forall \vv \Leftrightarrow \uu=0,\\
    b(\uu,\vv) &= 0 \quad \forall \uu \Leftrightarrow \vv=0.
  \end{align*}
\end{lemma}
\begin{proof}
  The implications ``$\Leftarrow$'' in the statements are obvious. Suppose that $b(\uu,\vv)=0$ for all
  $\vv\in V(\TT)$.
  This means that $\uu\in U(\TT)$ solves~\eqref{eq:b:1} with data $f=0$ and
  $(1/2-\dlo)u_0+\slo\phi_0=0$. Due to Lemma~\ref{lem:equi} \textit{(ii)},
  $(u,\wat\sigma|_\Gamma-\phi_0)\in H^1(\Omega)\times H^{-1/2}(\Gamma)$ is a solution
  of~\eqref{eq:jn} with the respective data $\GG=0$, $F=0$, $\rho=\phi_0$, and
  $\lambda = (1/2-\dlo)u_0$.
  We conclude that $(u,\wat\sigma|_\Gamma)$ solves~\eqref{eq:jn} with all right-hand side data
  equal to $0$. Lemma~\ref{lem:jn} shows $u=0$ and $\wat\sigma|_\Gamma=0$.
  Lemma~\ref{lem:equi} \textit{(ii)} finally shows that $\uu=0$.
  
  Now suppose $b(\uu,\vv)=0$ for all $\uu\in U(\TT)$. Testing with $\ssigma\in C_0^\infty(\el)^d$ shows
  $\ttau = -\pwnabla v$, and testing with $\wat\sigma$ shows that $v\in H^1(\Omega)$ and
  $v = \slo\psi$ on $\Gamma$.
  Furthermore, testing with appropriate $u\in C_0^\infty(\el)$ shows
  $\pwdiv\ttau=0$ on all $\el\in\TT$, 
  such that piecewise integration by parts yields
  \begin{align*}
    (\nabla u,\ttau)_\Omega = \dual{u}{\jump{\ttau\cdot\n}}_\cS
    = \dual{(1/2-\dlo)u}{\psi}_\Gamma = 0
  \end{align*}
  for all $u\in C_0^\infty(\Omega)$. Hence $\ttau\in\HH(\div,\Omega)$. 
  Now, as $\div \nabla v= 0$ in $\Omega$ and $v=\slo\psi$ on $\Gamma$, it holds $v = \slp\psi$
  in $\Omega$. This implies
  that on $\Gamma$ it holds
  \begin{align*}
    (1/2-\adlo)\psi=\ttau\cdot\n = -\frac{\partial}{\partial\n}V\psi = -(\adlo+1/2)\psi,
  \end{align*}
  where the last equation follows from the definition of $\adlo$. This shows that $\psi = -\psi$
  and hence $\psi=0$. It follows that $v=0$ and $\ttau=0$.

  With exactly the same reasoning one proves that $b$ is definite on $U(\Omega)\times V(\Omega)$.
\end{proof}
\begin{proof}[Proof of Theorems~\ref{thm:solve} and~\ref{thm:opt}]
  We will use Lemma~\ref{lem:dpg} with
  $U:=U(\TT)$, $V:=V(\TT)$, and $B: \uu \mapsto (\vv\mapsto b(\uu,\vv))$. First
  we check that the assumptions from Lemma~\ref{lem:basic} hold. Clearly, $U(\TT)$ and $V(\TT)$
  are reflexive Banach spaces with their respective norms.
  The operator $B$ is linear, bounded due to Lemma~\ref{lem:stab},
  injective due to Lemma~\ref{lem:definite}, and surjective due to Lemma~\ref{lem:surj}.
  Hence, Lemma~\ref{lem:dpg} is applicable.
  We obtain a unique solution $\uu\in U(\TT)$ with stability~\eqref{lem:dpg:stab}
  and best approximation~\eqref{lem:dpg:opt} in the norm $\norm{\cdot}{U(\TT),\opt}$.
  It remains to show the bounds
  \begin{align}\label{eq:unorm}
    \norm{\uu}{U(\Omega)} \lesssim \norm{\uu}{U(\TT),\opt} \lesssim \norm{\uu}{U(\TT)}
    \quad\text{ for all } \uu\in U(\TT).
  \end{align}

  We start with the upper bound. Inspection shows that $\norm{\cdot}{V(\TT),\opt}$ is the optimal test norm
  to $\norm{\cdot}{U(\TT)}$. Hence, we can use Lemma~\ref{lem:testnorms:equi:jn}
  and the left one of the identities~\eqref{lem:basic:norms} (the assumptions in Lemma~\ref{lem:basic}
  have been checked above) to conclude the upper bound in~\eqref{eq:unorm}.

  Now we show the lower bound in~\eqref{eq:unorm}.
  Due to Lemma~\ref{lem:adjoint:jn} and $V(\Omega)\subset V(\TT)$ we first obtain
  \begin{align}\label{eq:100}
    \sup_{\vv\in V(\Omega)}\frac{b(\uu,\vv)}{\norm{\vv}{V(\Omega)}} \lesssim
    \sup_{\vv\in V(\TT)}\frac{b(\uu,\vv)}{\norm{\vv}{V(\TT)}} = 
    \norm{\uu}{U(\TT),\opt}
  \end{align}
  Inspection shows that $\norm{\vv}{V(\Omega)}$ is the optimal test norm
  to $\norm{\cdot}{U(\Omega)}$. It remains to check that the left-hand side of~\eqref{eq:100}
  is indeed $\norm{\cdot}{U(\Omega)}$. To that end, we will again apply Lemma~\ref{lem:basic},
  but this time with $U:=U(\Omega)$, $V:=V(\Omega)$, and $B: \uu \mapsto (\vv\mapsto b(\uu,\vv))$.
  We check again that the assumptions hold. The spaces $U(\Omega)$ and $V(\Omega)$ are reflexive Banach
  spaces with their respective norms.
  The operator $B$ is linear and bounded due to Lemma~\ref{lem:stab}.
  Bijectivity follows from the Babu\v{s}ka-Brezzi theory which applies by
  Lemmas~\ref{lem:definite} and~\ref{lem:adjoint:jn}.
  Hence, by Lemma~\ref{lem:basic}, the identity on the left of~\eqref{lem:basic:norms}
  shows that the left-hand side of~\eqref{eq:100} is indeed $\norm{\cdot}{U(\Omega)}$.
  This shows the lower bound in~\eqref{eq:unorm} and
  concludes the proof of Theorems~\ref{thm:solve} and~\ref{thm:opt}.
\end{proof}
\section{Numerical experiments}\label{section:numerics}
We conducted two numerical experiments for $d=2$ to support our analysis. As partitions $\TT$ we choose
regular triangulations, i.e., all elements $T\in\TT$ are triangles, and there are no hanging nodes.
All triangulations are shape-regular and quasi-uniform; $h$ denotes the global mesh size and $N:=\#\TT$
the number of triangles, which satisfy $N^{-1/2} \simeq h$.
For $p\in\N$, denote by $P_p(\el)$ the space of polynomials of degree at most $p$ on an element
$\el$ and by $P_p(\ed)$ polynomials of degree at most $p$ on an edge $\ed$. Then define
\begin{align*}
  P_p(\TT) &:= \prod_{\el\in\TT} P_p(\el),\\
  P_p(\cS) &:= \left\{ u_p \mid u_p|_\ed \in P_p(\ed) \text{ for all edges } \ed \right\},\\
  S_p(\cS) &:= \left\{ u_p \mid u_p|_\ed \in P_p(\ed) \text{ for all edges } \ed \right\}
  \cap H^{1/2}(\cS).
\end{align*}
As discrete trial space, we use the conforming lowest-order space
\begin{align*}
  U_\hp(\TT) = U_0(\TT) := P_0(\TT)^d \times P_0(\TT) \times S_1(\cS) \times P_0(\cS).
\end{align*}
Theorem~\ref{thm:opt} and standard approximation theory combined with the definitions of the
norms $\norm{\wat u}{H^{1/2}(\cS)}$, $\norm{\wat\sigma}{H^{-1/2}(\cS)}$ by canonical traces,
cf.~\cite{BoffiBF_13,dg11}, shows that
\begin{align*}
  \norm{(\ssigma-\ssigma_\hp,u-u_\hp,&\wat u-\wat u_\hp,\wat\sigma-\wat\sigma_\hp)}{U(\Omega)}
  \lesssim h^s \Bigl(\norm{u}{H^{1+s}(\Omega)} + \norm{f}{H^{s}(\Omega)}\Bigr) \quad\text{ for } s\leq 1.
\end{align*}
The optimal test space $\Theta(U_0(\TT))$ has finite dimension, but its computation requires
to invert the Riesz map in $V = V(\TT)$, which has infinite dimension.
We will approximate the operator $\Theta$ in a finite-dimensional subspace $V_0(\TT)\subset V(\TT)$.
This approach is called \textit{practical} DPG method, cf.~\cite{gq14}.
We choose $V_0(\TT)$ to be
\begin{align*}
  V_0(\TT) := P_2(\TT) \times P_2(\TT)^2 \times P_1(\TT|_\Gamma),
\end{align*}
where $\TT|_\Gamma$ are the edges of the mesh $\TT$ on the boundary.
In both examples, we define a domain $\Omega$ and prescribe a function $u:\Omega\rightarrow\R$.
Then, with $u^c=0$, we solve~\eqref{eq:tp} with $u_0=u|_\Gamma$ and $\phi_0 = \partial_\nn u|_\Gamma$.

\paragraph{Experiment with smooth solution.}
In the first example, $\Omega := (-0.1,0.1)^2$ and $u(x,y)=\sin(\pi x)\sin(\pi y)$
is a smooth function. Hence, we expect and observe a convergence rate of $\OO(h)$.
\begin{figure}[ht]
  \centering
  \psfrag{enrgerr^2}[cr][cr]{\scriptsize$\norm{\uu-\uu_0}{U(\TT),\opt}^2$}
  \psfrag{L2_u^2}[cr][cr]{\scriptsize$\norm{u-u_0}{L_2(\Omega)}^2$}
  \psfrag{L2_sigma^2}[cr][cr]{\scriptsize$\norm{\ssigma-\ssigma_0}{\LL_2(\Omega)}^2$}
  \psfrag{N^{-1}}[cr][cr]{\scriptsize$N^{-1}$}
  \psfrag{N}[cr][cr]{\scriptsize $N$}
  \includegraphics[width=0.8\textwidth]{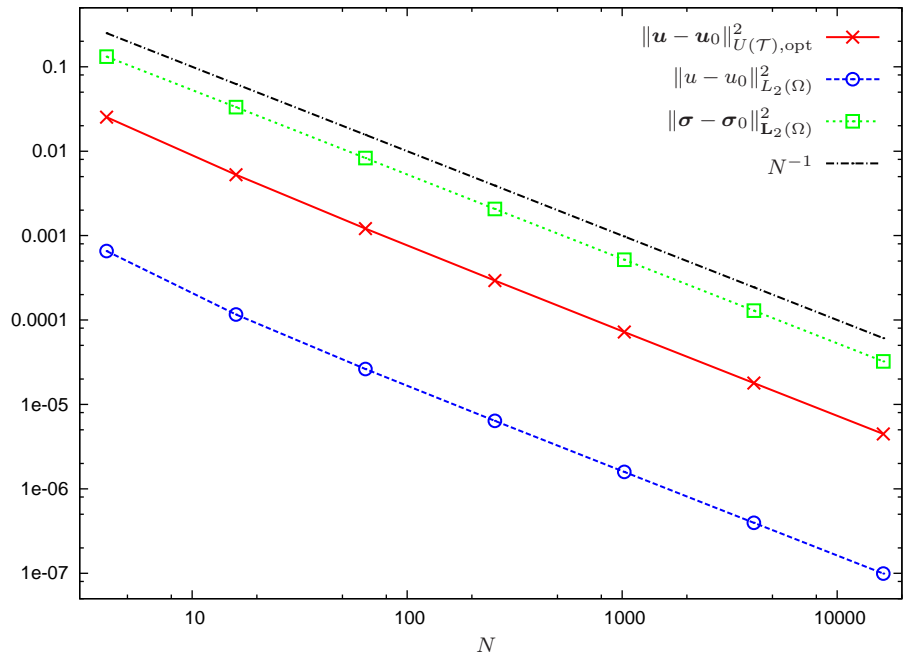}
  \caption{Squared error norms for experiment with smooth solution.}
  \label{fig:exp:smooth}
\end{figure}

\paragraph{Experiment with non-smooth solution.}
In the second example, $\Omega := (-0.25,0.25)^2\setminus (0,0.25)\times(-0.25,0)$ is an L-shaped
domain and $u(r,\phi) = r^{2/3}\sin(2\phi/3)$ with polar coordinates $(r,\phi)$ centered at
the origin. It follows that $u\in H^{s}(\Omega)$ for all $s<1+2/3$ and $f=0$ so that
we expect a convergence rate of $\OO(h^{2/3})$. The energy error $\norm{\uu-\uu_0}{U(\TT),\opt}$
as well as $\norm{\ssigma-\ssigma_0}{\LL_2(\Omega)}$ indeed have order $\OO(h^{2/3})$,
while we observe an improved order of $\OO(h)$ for $\norm{u-u_0}{L_2(\Omega)}$.
\begin{figure}[ht]
  \centering
  \psfrag{enrgerr^2}[cr][cr]{\scriptsize$\norm{\uu-\uu_0}{U(\TT),\opt}^2$}
  \psfrag{L2_u^2}[cr][cr]{\scriptsize$\norm{u-u_0}{L_2(\Omega)}^2$}
  \psfrag{L2_sigma^2}[cr][cr]{\scriptsize$\norm{\ssigma-\ssigma_0}{\LL_2(\Omega)}^2$}
  \psfrag{N^{-1}}[cr][cr]{\scriptsize$N^{-1}$}
  \psfrag{N^{-2/3}}[cr][cr]{\scriptsize$N^{-2/3}$}
  \psfrag{N}[cr][cr]{\scriptsize $N$}
  \includegraphics[width=0.8\textwidth]{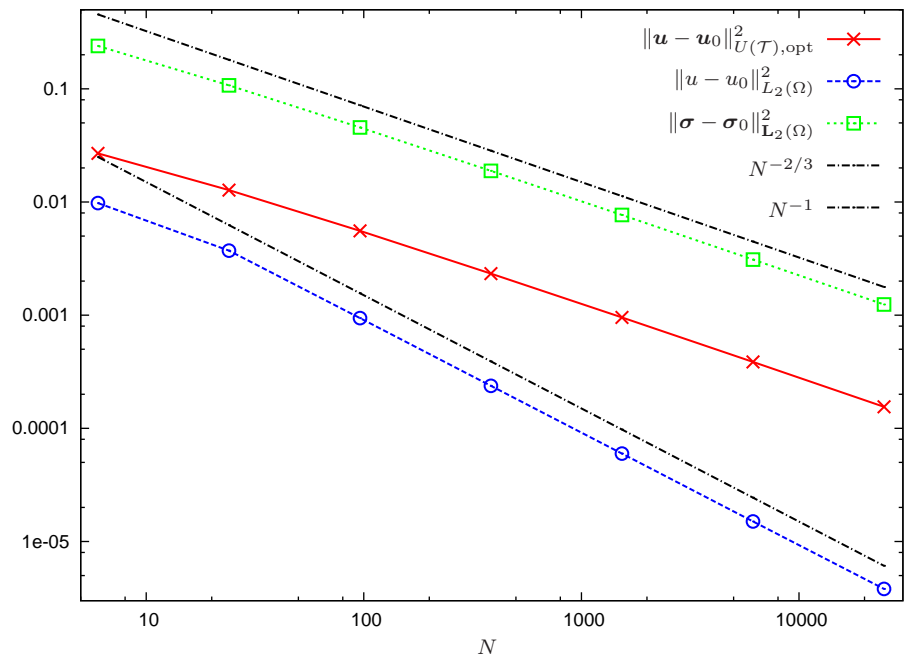}
  \caption{Squared error norms for experiment with non-smooth solution.}
  \label{fig:exp:nonsmooth}
\end{figure}
\section{Conclusion}\label{section:conclusio}
We presented a numerical method for a transmission problem. The method uses an ultra-weak (finite element)
formulation for the interior and standard boundary integral equations for the exterior. The whole
system is discretized by a discontinuous Petrov-Galerkin approach and optimal test functions.
We obtain quasi-optimality for the field variables
$u$ and $\ssigma$ as well as for the trace $\wat u$ and normal derivative $\wat\sigma$ on the interface $\Gamma$.
Our analysis builds on the unique solvability of the classical non-symmetric coupling of finite and boundary
elements. Numerical experiments support our analysis. 

We expect that our method can be extended to other PDEs in the interior for
which a DPG analysis is available, e.g., convection-diffusion~\cite{DemkowiczH_13}.
Also, based on recent results~\cite{ffkp14} on unique solvability of the non-symmetric coupling
for finite and boundary elements for elasticity problems and DPG finite elements for
linear elasticity~\cite{BramwellDGQ_12_elasticity}, we expect that our DPG strategy for transmission
problems can be extended to problems from linear elasticity.
\bibliographystyle{abbrv}
\bibliography{literature}
\end{document}